\documentclass[a4paper,11pt, reqno]{amsart}
\usepackage{amssymb,amsthm,amsmath}
\usepackage{cite}

\usepackage{amsmath,amsthm,amscd,amssymb}
\usepackage{latexsym}
\usepackage[colorlinks,citecolor=red,pagebackref,hypertexnames=false]{hyperref}

\theoremstyle{plain}
\newtheorem{theorem}{Theorem}[section]
\newtheorem{Def}[theorem]{Definition}
\newtheorem{lemma}[theorem]{Lemma}

\newtheorem{proposition}[theorem]{Proposition}

\theoremstyle{definition}

\theoremstyle{remark}

\newtheorem{case[theorem]}{Case}

\def \R{{\mathbb R}}
\def \N{{\mathbb N}}

\def \C{{\mathbb C}}

\def\H{{\mathbb H}}

\def\norm#1.#2.{\lVert#1\rVert_{#2}}

\def\R{\mathbb R}

\def \M{{\mathcal M}}
\def \W{{\mathcal W}}

\def \S{{\mathcal S}}

\def \H{{\mathcal H}}

\title[Qualitative uncertainty principles for the windowed Opdam--Cherednik transform ]{Qualitative uncertainty principles for the windowed Opdam--Cherednik transform on weighted modulation spaces
}

\author{Shyam Swarup Mondal} 
\author{Anirudha Poria}
\thanks{Research supported by ERC Starting Grant No. 713927.}

\address{ \endgraf Department of Mathematics, Indian Institute of Technology Guwahati, Guwahati 781039, India} 
\email{mondalshyam055@gmail.com}

\address{Department of Mathematics, Bar-Ilan University, Ramat-Gan 5290002, Israel}
\email{anirudhamath@gmail.com}


\date{\today}

\begin{document}
\maketitle
\begin{abstract} 
	The aim of this paper is to establish a few qualitative uncertainty principles for the windowed Opdam--Cherednik transform on weighted modulation spaces associated with this transform. In particular, we obtain the Cowling--Price's, Hardy's and Morgan's uncertainty principles for this transform on weighted modulation spaces. The proofs of the results are based on versions of the Phragm{\'e}n--Lindl{\"o}f type result for several complex variables on weighted modulation spaces and the properties of the Gaussian kernel associated with the Jacobi--Cherednik operator. 
\end{abstract}    
 
\section{Introduction} 

The classical uncertainty principle   states that a non-zero function and its Fourier transform cannot   both be sharply localized.  Several forms of the uncertainty principle can be formulated depending on various ways of measuring the localization of the function.  Mainly, there are  two types of uncertainty principle:  qualitative  and  quantitative uncertainty principles.   Qualitative uncertainty principles  imply the vanishing of a function under some strong conditions on the function. In particular,    Cowling and Price \cite{cow83}, Morgan \cite{mor34},  Hardy \cite{har33},  and Beurling   \cite{hor91}  theorems are    examples of qualitative uncertainty principles.  On the other side, quantitative uncertainty principles    tell  us  information about  how a function and its Fourier transform are related to each other.  For example,  Donoho and Stark \cite{don89},  Slepian and Pollak \cite{sle61},  and   Benedicks \cite{ben85}  theorems  are  quantitative uncertainty principles.

One of the celebrated uncertainty principles in harmonic analysis is Hardy's theorem \cite{har33}. 
This theorem is about the decay of a measurable function and its Fourier transform at infinity. More precisely, let $a$ and $b$ be two positive constants and suppose that $f$ is a measurable function on $\R$ such that 
\[ |f(x)| \leq C e^{-a  x^2}  \quad \mathrm{and} \quad  |\hat{f}(\xi)| \leq C e^{-b  \xi^2}, \]
for some constants $ C > 0$ and where $\hat{f}$ is the Fourier transform of $f$ formally defined by 
$$\hat{f}(\xi)=\int_{-\infty}^{\infty} f(t) e^{-2 \pi i  \xi t} dt .$$ 
Then $f = 0$ almost everywhere if $ab> \frac{1}{4}$,   $f(x)=C e^{-a  x^2}$ for some constant $C$ if $ab=\frac{1}{4}$, and  there are infinitely many non-zero functions satisfying the assumptions 
if $ab<\frac{1}{4}$. Later,  an $L^p$-version of  this  theorem was proved  by Cowling and Price in \cite{cow83}.  It states that:  let  $1 \leq p, q \leq \infty$ with  $\min(p, q)$ is finite, and $f$ be a measurable function on $\mathbb{R}$ such that
$$\Vert e^{a x^2} f  \Vert_p < \infty\quad \text{and}\quad\Vert e^{b \xi^2} \hat{f}  \Vert_q < \infty.$$
Then  $f=0$ almost everywhere if $ab \geq \frac{1}{4}$, and  there are infinitely many non-zero functions satisfying the assumptions if $ab<\frac{1}{4}$.

The Hardy's and Cowling--Price's uncertainty principles were extended to different settings by many authors (see \cite{tha04}). Particularly, Morgan in \cite{mor34} obtained the following uncertainty principle by replacing the   function $e^{a x^2}$  by   $e^{a |x|^\alpha}$, where $\alpha > 2$ in Hardy's theorem. It states that  for   two positive real numbers  $\alpha, \beta$ such that  $\alpha >2$ and $1/\alpha+1/\beta=1$, if
\[ e^{a|x|^\alpha} f \in L^\infty(\R) \quad \text{and} \quad  e^{b|\lambda|^\beta} \hat{f} \in L^\infty(\R), \] then $f = 0$ almost everywhere for  $ (a \alpha)^{1/\alpha} (b \beta)^{1/\beta} > \left( \sin \left( \frac{\pi}{2}(\beta -1 ) \right) \right)^{1/\beta}$. Further,  $L^p$--$L^q$-version of Morgan's theorem was proved   by Ben Farah and Mokni in \cite{far03}.  For a more detailed study on the history of the uncertainty principle, and   many other generalizations and variations of the uncertainty principle, we refer to the book of Havin and J\"oricke \cite{hav94},  and the excellent survey of Folland and Sitaram \cite{fol97}.

Considerable attention has been devoted to finding generalizations to new contexts for the Hardy's, Cowling--Price's, and Morgan's uncertainty principles.  For example, these theorems were investigated in \cite{mej16} for the generalized Fourier transform, and in \cite{thanga01} for the Heisenberg group. Further, an $L^p$ version of Hardy's theorem was proved for the Dunkl transform in \cite{gal04}. In \cite{dah12},  Daher et al. have obtained  some uncertainty principles for the Cherednik transform as a generalization of Euclidean uncertainty principles for the Fourier transform. These results are  further extended to the Opdam--Cherednik transform in \cite{mej014} using composition properties of the Opdam--Cherednik transform and classical uncertainty principles for the Fourier transform.  Moreover, these types of uncertainty principles for the Opdam--Cherednik transform  on   modulation spaces were studied by the second author in \cite{por21}.    Recently, the  second author   introduced  the  windowed  Opdam--Cherednik transform and discussed  the time-frequency analysis of localization operators associated  with this  transform  on modulation spaces  in \cite{por2021}.  Further, we have investigated some quantitative uncertainty principles for the windowed Opdam--Cherednik transform in \cite{shyam}. However, upto our knowledge, qualitative uncertainty principles for this  transform  have not been studied in weighted modulation spaces. In this paper, we extend   the Cowling--Price's, Hardy's, and Morgan's uncertainty principles for the windowed Opdam--Cherednik transform on weighted modulation spaces associated with this transform. 

A common key  to obtain uncertainty principles for the Opdam--Cherednik transform or some other generalized transform is to use the H\"older inequality and show that this transform is an entire function on $\mathbb{C}$ (see  \cite{por21}).  In the case of the windowed Opdam--Cherednik transform, the main difficulty is that, the time-frequency shift of the window function in the integral representation of this transform does not satisfy the exponential decay condition, whereas for the Opdam--Cherednik transform the eigenfunction in the integral representation  satisfies the decay condition.  To overcome this difficulty, we  consider  the non-zero window function $g$ from a suitable modulation  space  and  apply  H\"older's inequality to show that this  transform is an entire function on $\mathbb{C}^2.$ 

An important motivation to prove these types of  qualitative uncertainty principles for the windowed Opdam--Cherednik transform on weighted modulation spaces arises from the classical uncertainty principles for the Fourier transform on the Lebesgue spaces.   Over the years,  modulation spaces have become one of the most active branches of research in modern contemporary mathematics due to their  appearances in current topics such as pseudo-differential operators,  partial differential equations, etc.,  and used extensively in several areas of  analysis, engineering, and physics.   Uncertainty principles have implications in two main areas: quantum mechanics and signal analysis, and weighted modulation spaces are broadly used in these areas. We hope that the study of uncertainty principles for the weighted modulation spaces makes a significant impact in these areas. Another important motivation to study the Jacobi--Cherednik operators arises from their relevance in the algebraic description of exactly solvable quantum many-body systems of Calogero--Moser--Sutherland type (see \cite{die00, hik96}) and they provide a useful tool in the study of special functions with root systems (see \cite{dun92, hec91}). These describe algebraically integrable systems in one dimension and have gained significant interest in mathematical physics. Other motivation for the investigation of uncertainty principles for  the windowed Opdam--Cherednik transform is to generalize the previous subjects which are bound with the physics. For a more detailed study, we refer to \cite{mej16}. 

Since weighted modulation spaces are much larger spaces than the weighted  Lebesgue spaces, a natural question to ask is: can we determine the functions $f$
such that $f$ and the windowed Opdam--Cherednik transform of $f$ satisfying the conditions of Hardy's, Cowling--Price's, and    Morgan's theorems for the weighted modulation spaces?  In this paper, we give affirmative answers to all of  these questions.  The  natural key to obtaining extensions of uncertainty principles for the   Opdam--Cherednik transform is a slice formula, that is, this transform is decomposed as a composition of the classical Fourier transform and the Jacobi--Cherednik intertwining operator (see \cite{mej014}). However, without using a slice formula, we obtain  uncertainty principles for  the windowed Opdam--Cherednik transform by using an estimate for  the Gaussian kernel \cite{fit89}.   Since the standard weighted modulation spaces are not suited to this transform, here  we consider the weighted modulation spaces associated with this transform. We prove  uncertainty principles by using   versions of the Phragm{\'e}n--Lindl{\"o}f type result   for several complex variables on weighted modulation spaces and the properties of the Gaussian kernel associated with the Jacobi--Cherednik operator.    

Apart from introduction, the paper is organized as follows. In Section \ref{sec2}, we recall some basic facts about the Jacobi--Cherednik operator and give the main results for the Opdam--Cherednik transform. Then, we discuss the results related to the windowed Opdam--Cherednik transform and give some properties of the Gaussian kernel associated with the Jacobi--Cherednik operator. In Section \ref{sec3}, we study the weighted modulation spaces associated with the windowed Opdam--Cherednik transform. In Section \ref{sec4}, we establish a few qualitative uncertainty principles for this transform on weighted modulation spaces. First, we prove a version of the Phragm{\'e}n--Lindl{\"o}f type result for several complex variables, and using it, we obtain the Cowling--Price's theorem for the windowed Opdam--Cherednik transform on weighted modulation spaces. Then, we prove an analogue of the classical Hardy's theorem for this transform on weighted modulation spaces. Finally, we give another version of the Phragm{\'e}n--Lindl{\"o}f type result  and obtain Morgan's theorem for the windowed Opdam--Cherednik transform on weighted modulation spaces.

\section{Harmonic analysis and the windowed Opdam--Cherednik transform}\label{sec2}
In this section, we  collect the  necessary definitions and results from the harmonic analysis related to the windowed Opdam–Cherednik transform. For a detailed discussion on this transform, we refer  to \cite{por2021}. 

Let $T_{\alpha, \beta}$ denote the Jacobi--Cherednik differential--difference operator (also called the Dunkl--Cherednik operator)
\[T_{\alpha, \beta} f(x)=\frac{d}{dx} f(x)+ \Big[ 
(2\alpha + 1) \coth x + (2\beta + 1) \tanh x \Big] \frac{f(x)-f(-x)}{2} - \rho f(-x), \]
where $\alpha, \beta$ are two parameters satisfying $\alpha \geq \beta \geq -\frac{1}{2}$, $\alpha > -\frac{1}{2}$, and $\rho= \alpha + \beta + 1$. Let $\lambda \in \C$. The Opdam hypergeometric functions $G^{\alpha, \beta}_\lambda$ on $\R$ are eigenfunctions $T_{\alpha, \beta} G^{\alpha, \beta}_\lambda(x)=i \lambda  G^{\alpha, \beta}_\lambda(x)$ of $T_{\alpha, \beta}$ that are normalized such that $G^{\alpha, \beta}_\lambda(0)=1$. The eigenfunction $G^{\alpha, \beta}_\lambda$ is given by
\[G^{\alpha, \beta}_\lambda (x)= \varphi^{\alpha, \beta}_\lambda (x) - \frac{1}{\rho - i \lambda} \frac{d}{dx}\varphi^{\alpha, \beta}_\lambda (x)=\varphi^{\alpha, \beta}_\lambda (x)+ \frac{\rho+i \lambda}{4(\alpha+1)} \sinh 2x \; \varphi^{\alpha+1, \beta+1}_\lambda (x),  \]
where $\varphi^{\alpha, \beta}_\lambda (x)={}_2F_1 \left(\frac{\rho+i \lambda}{2}, \frac{\rho-i \lambda}{2} ; \alpha+1; -\sinh^2 x \right) $ is the classical Jacobi function. For any $ \lambda \in \C$ and $x \in  \R$, the eigenfunction
$G^{\alpha, \beta}_\lambda$ satisfy
$|G^{\alpha, \beta}_\lambda(x)| \leq C \; e^{-\rho |x|} e^{|\text{Im} (\lambda)| |x|},$
where $C$ is a positive constant. Since $\rho > 0$, we have
$
|G^{\alpha, \beta}_\lambda(x)| \leq C \; e^{|\text{Im} (\lambda)| |x|}. 
$

Let   $C_c (\R)$ denotes the space of continuous functions on $\R$ with compact support.  The Opdam--Cherednik transform is defined as follows.
\begin{Def}
Let $\alpha \geq \beta \geq -\frac{1}{2}$ with $\alpha > -\frac{1}{2}$. The Opdam--Cherednik transform $ \H_{\alpha, \beta} f$ of a function $f \in C_c(\R)$ is defined by
\[ \H_{\alpha, \beta} f (\lambda)=\int_{\R} f(x)\; G^{\alpha, \beta}_\lambda(-x)\; A_{\alpha, \beta} (x) dx \quad \text{for all } \lambda \in \C, \] 
where $A_{\alpha, \beta} (x)= (\sinh |x| )^{2 \alpha+1} (\cosh |x| )^{2 \beta+1}$. The inverse Opdam--Cherednik transform for a suitable function $g$ on $\R$ is given by
\[ \H_{\alpha, \beta}^{-1} g(x)= \int_{\R} g(\lambda)\; G^{\alpha, \beta}_\lambda(x)\; d\sigma_{\alpha, \beta}(\lambda) \quad \text{for all } x \in \R, \]
where $$d\sigma_{\alpha, \beta}(\lambda)= \left(1- \dfrac{\rho}{i \lambda} \right) \dfrac{d \lambda}{8 \pi |C_{\alpha, \beta}(\lambda)|^2}$$ and 
$$C_{\alpha, \beta}(\lambda)= \dfrac{2^{\rho - i \lambda} \Gamma(\alpha+1) \Gamma(i \lambda)}{\Gamma \left(\frac{\rho + i \lambda}{2}\right)\; \Gamma\left(\frac{\alpha - \beta+1+i \lambda}{2}\right)}, \quad \lambda \in \C \setminus i \mathbb{N}.$$
\end{Def}

The Plancherel formula is given by 
\begin{equation}\label{eq03}
\int_{\R} |f(x)|^2 A_{\alpha, \beta}(x) dx=\int_\R  \H_{\alpha, \beta} f(\lambda) \overline{\H_{\alpha, \beta} \check{f}(-\lambda)} \; d \sigma_{\alpha, \beta} (\lambda),
\end{equation}
where $\check{f}(x):=f(-x)$.

Let $L^p(\R,A_{\alpha, \beta} )$ (resp. $L^p(\R, \sigma_{\alpha, \beta} )$), $p \in [1, \infty] $, denote the $L^p$-spaces corresponding to the measure $A_{\alpha, \beta}(x) dx$ (resp. $d | \sigma_{\alpha, \beta} |(x)$). The Schwartz space $\S_{\alpha, \beta}(\R )=(\cosh x )^{-\rho} \S(\R)$ is defined as the space of all differentiable functions $f$ such that 
$$ \sup_{x \in \R} \; (1+|x|)^m e^{\rho |x|} \left|\frac{d^n}{dx^n} f(x) \right|<\infty,$$ 
for all $m, n \in \N_0 = \N \cup \{0\}$, equipped with the obvious seminorms. The Opdam--Cherednik transform $\H_{\alpha, \beta}$ and its inverse $\H_{\alpha, \beta}^{-1}$ are topological isomorphisms between the space $\S_{\alpha, \beta}(\R )$ and the space $\S(\R)$ (see \cite{sch08}, Theorem 4.1).

The generalized translation operator associated with the Opdam--Cherednik transform is defined by \cite{ank12}
\begin{equation}\label{eq044}
	\tau_x^{(\alpha, \beta)} f(y)=\int_{\R} f(z) \; {d\mu}_{\;x, y}^{(\alpha, \beta)}(z),
\end{equation}
where ${d\mu}_{\;x, y}^{(\alpha, \beta)}$ is given by 
\begin{equation}\label{eq05}
	{d\mu}_{\;x, y}^{(\alpha, \beta)}(z)=
	\begin{cases} 
		\mathcal{K}_{\alpha, \beta}(x,y,z)\; A_{\alpha, \beta}(z)\; dz & \text{if} \;\; xy \neq 0 \\
		d \delta_x (z) & \text{if} \;\; y=0 \\
		d \delta_y (z) & \text{if} \;\; x=0
	\end{cases}  
\end{equation}
and
\begin{equation*}
	\begin{aligned}
		\mathcal{K}_{\alpha, \beta} {} & (x,y,z) 
		=M_{\alpha, \beta} |\sinh x \cdot \sinh y \cdot \sinh z  |^{-2 \alpha} \int_0^\pi g(x, y, z, \chi)_+^{\alpha-\beta-1} 
		\\
		& \times \left[ 1 - \sigma^\chi_{x,y,z}+ \sigma^\chi_{x,z,y} + \sigma^\chi_{z,y,x} + \frac{\rho}{\beta+\frac{1}{2}} \coth x \cdot \coth y \cdot \coth z (\sin \chi)^2 \right] \times (\sin \chi)^{2 \beta}\; d \chi,
\end{aligned}
\end{equation*}
where
$$ M_{\alpha, \beta}=\frac{\Gamma(\alpha+1)}{\sqrt{\pi} \Gamma(\alpha-\beta) \Gamma\left(\beta+\frac{1}{2}\right)}, $$ 
if $x, y, z \in \R \setminus \{0\} $ satisfy the triangular inequality  $||x|-|y||<|z|<|x|+|y|$, and $\mathcal{K}_{\alpha, \beta} (x,y,z) =0$ otherwise. Here 
\[ \sigma^\chi_{x,y,z}=
\begin{cases} 
	\frac{\cosh x \cdot \cosh y - \cosh z \cdot \cos \chi}{\sinh x \cdot \sinh y} & \text{if} \;\; xy \neq 0 \\
	0 & \text{if} \;\; xy = 0 
\end{cases}
\quad \text{for} \; x, y, z \in \R, \; \chi \in [0,\pi], \]
$g(x, y, z, \chi) = 1- \cosh^2 x - \cosh^2 y - \cosh^2 z + 2 \cosh x \cdot \cosh y \cdot \cosh z \cdot \cos \chi$, and 
\[ g_+=
\begin{cases} 
	g &  \text{if} \;\;  g> 0 \\
	0 & \text{if} \;\;  g \leq 0.
\end{cases} \]
The kernel $\mathcal{K}_{\alpha, \beta} (x,y,z)$ satisfies the following symmetry properties:
$$\mathcal{K}_{\alpha, \beta} (x,y,z)=\mathcal{K}_{\alpha, \beta} (y,x,z), \mathcal{K}_{\alpha, \beta} (x,y,z)=\mathcal{K}_{\alpha, \beta} (-z,y,-x), \mathcal{K}_{\alpha, \beta} (x,y,z)=\mathcal{K}_{\alpha, \beta} (x,-z,-y).$$

For every $x,y \in \R$, we have
$
	\tau_x^{(\alpha, \beta)} f(y) = \tau_y^{(\alpha, \beta)} f(x),
$
and 
$
	\H_{\alpha, \beta}(\tau_x^{(\alpha, \beta)} f )(\lambda)=G_\lambda^{\alpha, \beta} (x) \; \H_{\alpha, \beta} (f)(\lambda),
$
for $f \in C_c(\R)$. For a more detailed study on the Opdam--Cherednik transform, we refer to \cite{sch08, opd00}.

Let $g \in L^2(\R,A_{\alpha, \beta} )$ and $\xi \in \R$, the modulation operator of $g$ associated with the Opdam--Cherednik transform is defined by 
$$
	\M^{(\alpha, \beta)}_\xi g={\H_{\alpha, \beta}}^{-1} 
	 \left( \sqrt{\tau_\xi^{(\alpha, \beta)} 
	 	|{\H_{\alpha, \beta}} (g)|^2 } 
 	\right) .
$$
Then, for any $g \in L^2(\R,A_{\alpha, \beta} )$ and $\xi \in \R$,   using the Plancherel formula (\ref{eq03}) and 
the translation invariance of the Plancherel measure 
$d \sigma_{\alpha, \beta}$,  we get 
$
	 \Vert \M^{(\alpha, \beta)}_\xi g  \Vert_{L^2(\R,A_{\alpha, \beta} )}=\left\Vert g \right\Vert_{L^2(\R,A_{\alpha, \beta} )}.
$
Now, for a non-zero window function $g \in L^2(\R,A_{\alpha, \beta} )$ and $(x, \xi) \in \R^2$, we define the function $g_{x, \xi}^{(\alpha, \beta)}$   by 
\begin{equation*}
	g_{x, \xi}^{(\alpha, \beta)}= \tau_x^{(\alpha, \beta)} \M^{(\alpha, \beta)}_\xi g.
\end{equation*}
For any   $f \in L^2(\R,A_{\alpha, \beta} )$,  the windowed Opdam--Cherednik transform is defined by 
\begin{equation}\label{eq13}
\W^{(\alpha, \beta)}_g(f)(x,\xi)=\int_\R f(s) \; \overline{g_{x, \xi}^{(\alpha, \beta)}(-s)} \;  A_{\alpha, \beta}(s) \; ds, \quad (x,\xi) \in \R^2.
\end{equation}
We define the measure $A_{\alpha, \beta} \otimes \sigma_{\alpha, \beta}$ on $\R^2$ by 
\begin{equation*}
d(A_{\alpha, \beta} \otimes \sigma_{\alpha, \beta})(x, \xi)= A_{\alpha, \beta}(x)  dx \; d|\sigma_{\alpha, \beta}|(\xi).
\end{equation*}
The windowed Opdam--Cherednik transform satisfies the following properties (see \cite{por2021}). 
\begin{proposition}

\begin{enumerate}
	\item $($Plancherel's formula$)$  Let $g \in L^2(\R,A_{\alpha, \beta} )$ be a non-zero window function.  Then  for every $f \in L^2(\R,A_{\alpha, \beta})$,  we have 
	\begin{equation}\label{eq17}
		\left\Vert \W^{(\alpha, \beta)}_g(f) \right\Vert_{L^2(\R^2,  A_{\alpha, \beta}\otimes \sigma_{\alpha, \beta})} = \Vert f \Vert_{L^2(\R,A_{\alpha, \beta})} \; \Vert g \Vert_{L^2(\R,A_{\alpha, \beta})}.
	\end{equation}
\item $($Reconstruction formula$)$ Let $g \in L^2(\R,A_{\alpha, \beta} )$ be a non-zero positive window function. Then for every $F \in L^2(\R^2,A_{\alpha, \beta} \otimes \sigma_{\alpha, \beta})$, we have
\[{\W^{(\alpha, \beta)}_g}^{-1}(F)(\cdot )= \frac{1}{\Vert g \Vert^2_{L^2(\R,A_{\alpha, \beta})}}  \iint_{\R^2} F(x, \xi)\; g_{x, \xi}^{(\alpha, \beta)} (- \;\cdot) \; d(A_{\alpha, \beta} \otimes \sigma_{\alpha, \beta}) (x, \xi), \]
weakly in $L^2(\R,A_{\alpha, \beta} )$. 
\end{enumerate}	
\end{proposition}

Let $t > 0$. The Gaussian kernel $E^{\alpha, \beta}_t$ associated with the Jacobi--Cherednik operator is defined by
\begin{equation}\label{eq04}
	E^{\alpha, \beta}_t(s)={\W^{(\alpha, \beta)}_g}^{-1}(e^{-t (\lambda^2+\mu^2)})(s), \quad \text{for all } s\in \R.
\end{equation}
For all $t > 0$, $E^{\alpha, \beta}_t$ is an $C^\infty$-function on $\R$. Moreover, for all $t > 0$ and all $\lambda, \mu \in \R$, we have
\begin{equation}\label{eq05}
	\W^{(\alpha, \beta)}_g(E^{\alpha, \beta}_t) (\lambda, \mu)=e^{-t (\lambda^2+\mu^2)}.
\end{equation}
We refer to \cite{cho03} for further properties of the Gaussian kernel $E^{\alpha, \beta}_t$. From (\cite{fit89}, Theorem 3.1), there exist two real numbers $\mu_1$ and $\mu_2$, such that
\begin{equation}\label{eq06}
	\frac{e^{\mu_1 t}}{2^{2\alpha+1} \Gamma(\alpha+1) t^{\alpha+1}} \frac{e^{-\frac{x^2}{4t}}}{\sqrt{B_{\alpha, \beta}(x)}} \leq E^{\alpha, \beta}_t(x) \leq  \frac{e^{\mu_2 t}}{2^{2\alpha+1} \Gamma(\alpha+1) t^{\alpha+1}} \frac{e^{-\frac{x^2}{4t}}}{\sqrt{B_{\alpha, \beta}(x)}}, \quad \forall x \in \R,
\end{equation}
where $B_{\alpha, \beta} (x)= (\sinh |x|/|x| )^{2 \alpha+1} (\cosh |x| )^{2 \beta+1}$ for all $x \in \R \setminus \{0\}$ and $B_{\alpha, \beta}(0)=1$. Further, we have $A_{\alpha, \beta}(x)= |x|^{2\alpha+1} B_{\alpha, \beta}(x)$ and for all $x \in \R$, $B_{\alpha, \beta} (x) \geq 1$.

\section{Weighted modulation spaces associated with the  windowed Opdam--Cherednik transform}\label{sec3}

For $x, \xi \in \R$, let $M_\xi$ and $T_x$ denote the operators of modulation and translation. Then, the short-time Fourier transform (STFT)  of a function $f$ with respect to a window function $g \in  \S(\R)$ is defined by
\[ V_g f (x,\xi)=\langle f, M_\xi T_x g \rangle=\int_{\R} f(t) \overline{g(t-x)} e^{-2 \pi i \xi t} dt, \quad (x,\xi) \in \R^2. \] 

The modulation spaces were introduced by Feichtinger \cite{fei03, fei97}, by imposing integrability conditions on the STFT of tempered distributions.
 Here, we are interested in weighted modulation spaces with respect to the measure $A_{\alpha, \beta} \otimes \sigma_{\alpha, \beta}.$ We define the measure $A_{\alpha, \beta} \otimes  A_{\alpha, \beta}$  on $\R^2$ by $d(A_{\alpha, \beta} \otimes A_{\alpha, \beta})(x, \xi)=A_{\alpha, \beta}(x)dx\;A_{\alpha, \beta}(\xi)d\xi$.  
\begin{Def}
 Let $m $ be a non-negative function on $\mathbb{R}^{2}$,  $g \in \mathcal{S}(\R)$ be a  fixed non-zero window function, and $1 \leq p,q \leq \infty$. Then the weighted modulation space $M_m^{p,q}(\R, A_{\alpha, \beta})$  consists of all tempered distributions $f \in \mathcal{S'}(\R)$ such that $V_g f \in L_m^{p,q}(\R^2, A_{\alpha, \beta} \otimes  A_{\alpha, \beta})$. The norm on $M_m^{p,q}(\R, A_{\alpha, \beta})$ is 
\begin{eqnarray*}
\Vert f \Vert_{M_m^{p,q}(\R, A_{\alpha, \beta})}
&=& \Vert V_g f \Vert_{L_m^{p,q}(\R^2, A_{\alpha, \beta} \otimes  A_{\alpha, \beta})} \\
&=& \bigg( \int_{\R} \bigg( \int_{\R} |V_gf(x,\xi)|^p |m(x,\xi)|^p A_{\alpha, \beta}(x) dx \bigg)^{q/p} A_{\alpha, \beta}(\xi) d\xi \bigg)^{1/q} < \infty,
\end{eqnarray*}
with the usual adjustments if $p$ or $q$ is infinite.
\end{Def}
 If $p=q$, then we write $M_m^p(\R, A_{\alpha, \beta})$ instead of $M_m^{p,p}(\R, A_{\alpha, \beta})$. When $m=1$ on $\R^2$, then we write $M^{p, q}(\R, A_{\alpha, \beta})$ and $M^p(\R, A_{\alpha, \beta})$  for   $M_m^{p, q}(\R, A_{\alpha, \beta})$ and   $M_m^p(\R, A_{\alpha, \beta})$  respectively. 
Also,  we denote by $M_m^p(\R, \sigma_{\alpha, \beta})$ the weighted modulation space corresponding to the measure $d |\sigma_{\alpha, \beta}|(x)$ and $M_m^p(\R)$ the  weighted modulation space corresponding to the Lebesgue measure $dx$.

We define the measure $(A_{\alpha, \beta} \otimes \sigma_{\alpha, \beta})* (A_{\alpha, \beta} \otimes \sigma_{\alpha, \beta})$ on $\R^2\times \R^2$ by 
\begin{align*} 
d((A_{\alpha, \beta} \otimes \sigma_{\alpha, \beta})* (A_{\alpha, \beta} \otimes \sigma_{\alpha, \beta}))((x_1, \xi_1), (x_2, \xi_2))= 	d(A_{\alpha, \beta} \otimes \sigma_{\alpha, \beta} )(x_1, \xi_1 )\;	d(A_{\alpha, \beta} \otimes \sigma_{\alpha, \beta} )  (x_2, \xi_2).
\end{align*}

 Then $M_m^{p,q}(\R^2, A_{\alpha, \beta} \otimes \sigma_{\alpha, \beta})$ denotes the weighted modulation space  on $\mathbb{R}^2$  with respect to the measure  $A_{\alpha, \beta} \otimes \sigma_{\alpha, \beta} $. The norm on $M_m^{p,q}(\R^2, A_{\alpha, \beta} \otimes \sigma_{\alpha, \beta})$ is given by 
\begin{align*}
&	\Vert f \Vert_{M_m^{p,q}(\R^2, A_{\alpha, \beta} \otimes \sigma_{\alpha, \beta})}=\Vert V_g f \Vert_{L_m^{p,q}(\R^4, (A_{\alpha, \beta} \otimes \sigma_{\alpha, \beta})* (A_{\alpha, \beta} \otimes \sigma_{\alpha, \beta}))} 
	\\&= \bigg( \iint_{\R^2} \bigg( \iint_{\R^2} |V_gf((x_1, \xi_1), (x_2, \xi_2))|^p |m((x_1, \xi_1), (x_2, \xi_2))|^p \; d(A_{\alpha, \beta} \otimes \sigma_{\alpha, \beta})(x_1, \xi_1)\bigg)^{q/p} \\&\qquad \times \; d(A_{\alpha, \beta} \otimes \sigma_{\alpha, \beta})(x_2,\xi_2)  \bigg)^{1/q} < \infty,
\end{align*}
with the usual modification when $p=\infty$ or $q=\infty .$

 The definitions of $M_m^{p, q}(\R, A_{\alpha, \beta})$  and  $M_m^{p,q}(\R^2, A_{\alpha, \beta} \otimes \sigma_{\alpha, \beta})$ are  independent of the choice of $g$ in the sense that each different choice of $g$ defines   equivalent norms on $M_m^{p, q}(\R, A_{\alpha, \beta})$   and  $M_m^{p,q}(\R^2, A_{\alpha, \beta} \otimes \sigma_{\alpha, \beta})$ respectively. Each weighted modulation space is a Banach space. For $p=q=2$, we have   $M_m^2(\R, A_{\alpha, \beta}) =L_m^2(\R, A_{\alpha, \beta}).$ For other $p=q$, the space $M_m^p(\R, A_{\alpha, \beta})$ is not $L_m^p(\R, A_{\alpha, \beta})$. In fact for $p=q>2$, the space $M_m^p(\R, A_{\alpha, \beta})$ is a superset of $L_m^2(\R, A_{\alpha, \beta})$. We have the following inclusion 
\[ \mathcal{S}(\R) \subset M_m^1(\R, A_{\alpha, \beta}) \subset M_m^2(\R, A_{\alpha, \beta})=L_m^2(\R, A_{\alpha, \beta}) \subset M_m^\infty(\R, A_{\alpha, \beta}) \subset \mathcal{S'}(\R). \]
In particular, we have $M_m^p(\R, A_{\alpha, \beta}) \hookrightarrow L_m^p(\R, A_{\alpha, \beta})$ for $1 \leq p \leq 2$, and  $L_m^p(\R, A_{\alpha, \beta}) \hookrightarrow M_m^p(\R, A_{\alpha, \beta})$ for $2 \leq p \leq \infty$. Moreover, the dual of a weighted modulation space is also a weighted modulation space, if $p < \infty$, $q < \infty$, $(M_m^{p, q}(\R, A_{\alpha, \beta}))^{'} =M_{\frac{1}{m}}^{p', q'}(\R, A_{\alpha, \beta})$, where $p', \; q'$ denote the dual exponents of $p$ and $q$, respectively. For further properties and uses of weighted modulation spaces, we refer to \cite{gro01}.

 \section{Qualitative uncertainty principles for the windowed Opdam--Cherednik transform}\label{sec4}
 In this section, we obtain the Cowling--Price's, Hardy's and Morgan's uncertainty principles for  the  windowed Opdam--Cherednik transform on  weighted modulation spaces associated with this transform. From onwards, we consider  the weight function $m$ such that $m(x, \xi)\geq 1$  on $\mathbb{R}^{2}$  (resp. $\mathbb{R}^{4}$). We begin with the following lemma.
 \begin{lemma}\label{lem3}
 Let  $f(t_1, t_2)=1$ and $g(t_1, t_2)=e^{-\pi (t_1^2+t_2^2)}$. Then
 	\[ V_g f ((x_1, \xi_1), (x_2, \xi_2)) =e^{- 2 \pi i (x_1x_2+\xi_1\xi_2)} \; e^{-\pi (x_2^2+\xi_2^2)}. \]
 	Also, for $p\in [1, \infty)$ and $\rho_1, \rho_2, \sigma >0$, we have
 	\[ \|f\|_{M_{\frac{1}{m}}^p([\rho_1 \sigma, \; \rho_1 (\sigma+1)]\times [\rho_2 \sigma, \; \rho_2 (\sigma+1)])} \leq  \rho_1^{\frac{2}{p}}\rho_2^{\frac{2}{p}}.  \]
 \end{lemma}
 \begin{proof} Using the definition of the STFT, we get 
 	\begin{align*}
 		V_g f ((x_1, \xi_1), (x_2, \xi_2)) 
 		&= \int_{\R}\int_{\R} e^{-\pi [(t_1-x_1)^2+(t_2-\xi_1)^2]}~ e^{- 2 \pi i (x_2t_1+\xi_2t_2)} \;dt_1\;dt_2 \\
 		&= \int_{\R}\int_{\R} e^{-\pi  (s_1^2+s_2^2)}~ e^{- 2 \pi i [x_2(s_1+x_1)+\xi_2(s_2+\xi_1)]} \;ds_1\;ds_2 \\
 		&=e^{- 2 \pi i (x_1x_2+\xi_1\xi_2)} \int_{\R}\int_{\R} e^{-\pi  (s_1^2+s_2^2)}~ e^{- 2 \pi i (x_2s_1 +\xi_2s_2)} \;ds_1\;ds_2 \\
 		&=e^{- 2 \pi i (x_1x_2+\xi_1\xi_2)} \; e^{-\pi (x_2^2+\xi_2^2)}.
 	\end{align*}
 	Further,  we have
 	\begin{align*}
 		&\|f\|_{M_{\frac{1}{m}}^p([\rho_1 \sigma, \; \rho_1 (\sigma+1)]\times [\rho_2 \sigma, \; \rho_2 (\sigma+1)])} \\
 		&= \| V_g f \|_{L_{\frac{1}{m}}^p([\rho_1 \sigma, \; \rho_1 (\sigma+1)]\times [\rho_2 \sigma, \; \rho_2 (\sigma+1)]\times [\rho_1 \sigma, \; \rho_1 (\sigma+1)]\times [\rho_2 \sigma, \; \rho_2 (\sigma+1)])} \\
 		&\leq  \left( \int_{\rho_1  \sigma}^{\rho_1  (\sigma+1)} \int_{\rho_2 \sigma}^{\rho_2 (\sigma+1)} \int_{\rho_1  \sigma}^{\rho_1  (\sigma+1)} \int_{\rho_2 \sigma}^{\rho_2 (\sigma+1)} e^{-\pi p (x_2^2+\xi_2^2)}\;dx_1\;d\xi_1 \;dx_2\;d\xi_2 \right)^{\frac{1}{p}}\\
 		&\leq \left( \int_{\rho_1  \sigma}^{\rho_1  (\sigma+1)} \int_{\rho_2 \sigma}^{\rho_2 (\sigma+1)} \int_{\rho_1  \sigma}^{\rho_1  (\sigma+1)} \int_{\rho_2 \sigma}^{\rho_2 (\sigma+1)}  \;dx_1\;d\xi_1 \;dx_2\;d\xi_2 \right)^{\frac{1}{p}}\\
 		&=  \rho_1^{\frac{2}{p}}\rho_2^{\frac{2}{p}}. 
 	\end{align*}
 \end{proof}

\subsection{Cowling--Price's theorem for the windowed Opdam--Cherednik transform} 

In this subsection, we obtain an $M_m^p$--$M_m^q$-version of Cowling--Price's theorem for the windowed Opdam--Cherednik transform.  First, we establish the following lemma of Phragm{\'e}n--Lindl{\"o}f  type using a similar technique as in \cite{cow83}. This lemma plays a crucial role in the proof of Cowling--Price's theorem.  An $L^p$-version of the following lemma proved in \cite{cow83},  however here we prove the lemma for the weighted modulation space $M_m^p(\R^2, A_{\alpha, \beta} \otimes \sigma_{\alpha, \beta})$.

\begin{lemma}\label{lem4}
Let   $\Phi$ be analytic in the region $ Q = \{(r_1e^{i \theta_1}, r_2e^{i \theta_2}) : r_1, r_2>0, \;0 < \theta_1, \theta_2 < \frac{\pi}{2}\}$ and continuous on the closure $\bar{Q}$ of $Q$. Assume  that for any $p\in [1, \infty)$ and constants $A, a > 0$, we have 
\[|\Phi(x + iy, u+i v)| \leq A \; e^{a(x^2+u^2)} \quad \text{for} \;\; x + iy, u+iv  \in \bar{Q} , \]
and
\[  \| {\Phi}_{|\R^2} \|_{M_m^p(\R^2, A_{\alpha, \beta} \otimes \sigma_{\alpha, \beta})} \leq A. \]
Then
\[ \int_\sigma^{\sigma+1} \int_\sigma^{\sigma+1} |\Phi(\rho_1 e^{i \psi_1}, \rho_2 e^{i \psi_2} )| \; d\rho_1 d\rho_2 \leq  A \; \max \left\{e^{2a} , (\sigma + 1)^{\frac{4}{p}-2} \right\} \]
for $\psi_1, \psi_2 \in [0, \frac{\pi}{2}]$ and $\sigma \in \R^+$. 
\end{lemma}
\begin{proof}
Using the definition of $M_m^p(\R^2, A_{\alpha, \beta} \otimes \sigma_{\alpha, \beta})$ and the fact that there is a constant $k_1 > 0$ such that $|C_{\alpha, \beta} (\lambda)|^{-2} \geq k_1 |\lambda|^{2 \alpha +1}$ for all $\lambda \in \R$  with $|\lambda| \geq 1$ (see \cite{tri97}, page 157), and $A_{\alpha, \beta}(x)\geq 1$ for any $x\in \R$, we get
\begin{align}\label{eq31}\nonumber
& \| {\Phi}_{|\R^2} \|^p_{M_m^p(\R^2, A_{\alpha, \beta} \otimes \sigma_{\alpha, \beta})}= \| V_h \Phi \|^p_{L_m^p(\R^4, (A_{\alpha, \beta} \otimes \sigma_{\alpha, \beta})*( A_{\alpha, \beta} \otimes \sigma_{\alpha, \beta}))} \\\nonumber
& \geq   \int_{\R} \int_{\R} \int_{|\xi_1| \geq 1} \int_{|\xi_2| \geq 1} |V_h \Phi((x_1, \xi_1),( x_2, \xi_2))|^p \;|m((x_1, \xi_1),( x_2, \xi_2))|^p\\\nonumber
&\qquad \times \; A_{\alpha, \beta}(x_1) dx_1 \;A_{\alpha, \beta}(x_2) dx_2\; d |\sigma_{\alpha, \beta}|(\xi_1) \; d |\sigma_{\alpha, \beta}|(\xi_2) \\\nonumber
& \geq  \int_{\R} \int_{\R} \int_{|\xi_1| \geq 1} \int_{|\xi_2| \geq 1} |V_h \Phi((x_1, \xi_1),( x_2, \xi_2))|^p\; |m((x_1, \xi_1),( x_2, \xi_2))|^p\; dx_1 \; dx_2\\\nonumber&\qquad \times  \;\left|1-\frac{\rho}{i\xi_1}\right| \frac{d\xi_1}{8 \pi |C_{\alpha, \beta} (\xi_1)|^2}  \; \left|1-\frac{\rho}{i\xi_2}\right| \frac{d\xi_2}{8 \pi |C_{\alpha, \beta} (\xi_2)|^2} \\\nonumber
& \geq \frac{k_1^2k_2^2}{64\pi^2}  \int_{\R} \int_{\R} \int_{|\xi_1| \geq 1} \int_{|\xi_2| \geq 1} |V_h \Phi((x_1, \xi_1),( x_2, \xi_2))|^p\;|m((x_1, \xi_1),( x_2, \xi_2))|^p \\\nonumber&\qquad \times \; dx_1 \; dx_2\;   |\xi_1|^{2 \alpha +1}  \; |\xi_2|^{2 \alpha +1} \;d\xi_1 \;d\xi_2 \\
& \geq \frac{k_1^2k_2^2}{64\pi^2}  \int_{\R} \int_{\R} \int_{|\xi_1| \geq 1} \int_{|\xi_2| \geq 1} |V_h \Phi((x_1, \xi_1),( x_2, \xi_2))|^p \;|m((x_1, \xi_1),( x_2, \xi_2))|^p\nonumber\\&\qquad\times \; dx_1 \; dx_2 \;  d\xi_1 \;d\xi_2,
\end{align}
where $h \in \mathcal{S}(\R^2)$. This shows that $\| {\Phi}_{|\R^2} \|_{M_m^p(\R^2 )}   \leq A$. Next, we define a function $f$ on $\bar{Q}$ by $$ f(z_1, z_2)=\Phi(z_1, z_2) \exp \left(i \varepsilon e^{i \varepsilon} \big( z_1^{ (\pi - 2 \varepsilon)/\theta}+z_2^{ (\pi - 2 \varepsilon)/\theta}\big)+i a \cot (\theta) \left( z_1^2+z_2^2\right) /2 \right),$$ for $\theta \in (0, \pi/2)$ and $\varepsilon \in (0, \pi/2-\theta)$. 
 Then for $\psi_1, \psi_2\in [0, \theta]$, we get 
\begin{align*}
	&|	f(\rho_1e^{i\psi_1}, \rho_2e^{i\psi_2})|\\
	&\leq A \exp \left\{a \cos ^{2}(\psi_1) \rho_1^{2}-\varepsilon \sin (\varepsilon+(\pi-2 \varepsilon) \psi_1 / \theta) \rho_1^{(\pi-2 \varepsilon) / \theta}-a\cot (\theta) \sin (2 \psi_1) \rho_1^{2} / 2\right\}\\
	&\qquad \times \exp \left\{a \cos ^{2}(\psi_2) \rho_2^{2}-\varepsilon \sin (\varepsilon+(\pi-2 \varepsilon) \psi_2 / \theta) \rho_2^{(\pi-2 \varepsilon) / \theta}-a \cot (\theta) \sin (2 \psi_2) \rho_2^{2} / 2\right\}\\
	&\leq A \exp \left\{a (\rho_1^{2}+\rho_2^{2})-\varepsilon \sin (\varepsilon+(\pi-2 \varepsilon) \psi_1 / \theta) \rho_1^{(\pi-2 \varepsilon) / \theta}-\varepsilon \sin (\varepsilon+(\pi-2 \varepsilon) \psi_2 / \theta) \rho_2^{(\pi-2 \varepsilon) / \theta}\right\}.
\end{align*}

Applying a  similar approach as in \cite{cow83} to the function $f$, we get the subsequent estimates.

For $\rho_1,  \rho_2> 0$, $|	f(\rho_1e^{i\psi_1}, \rho_2e^{i\psi_2})|\leq A$ and thus we obtain
\begin{align*}
	\int_\sigma^{\sigma+1} \int_\sigma^{\sigma+1} |f(\rho_1 \tau_1, \rho_2 \tau_2)| \;d \tau_1d \tau_2
	\leq\int_\sigma^{\sigma+1} \int_\sigma^{\sigma+1} A\; d \tau_1d \tau_2=A.
\end{align*}
Similarly, if $\rho_1, \rho_2\in [0, (\sigma+1)^{-1}]$, then 
\begin{align*}
	\int_\sigma^{\sigma+1} \int_\sigma^{\sigma+1} |f(\rho_1 \tau_1, \rho_2 \tau_2)| \;d \tau_1d \tau_2
	\leq \sup\{ |f(\rho_1, \rho_2  )|:\rho_1, \rho_2\leq 1 \}\leq A e^{2a}.
\end{align*}
Finally,  for $\rho_1,  \rho_2> (\sigma+1)^{-1}$,  using H{\"o}lder's inequality for $M_m^p$ and Lemma \ref{lem3},  we get  
\begin{align*}
	&\int_\sigma^{\sigma+1} \int_\sigma^{\sigma+1} |f(\rho_1 \tau_1, \rho_2 \tau_2)| \;d \tau_1d \tau_2\\
	&= \frac{1}{\rho_1\rho_2} \int_{\rho_1 \sigma}^{\rho_1 (\sigma+1)}  \int_{\rho_2 \sigma}^{\rho_2 (\sigma+1)} |f( \tau_1, \tau_2)|\; d \tau_1d \tau_2\\
	&\leq  \frac{1}{\rho_1\rho_2}\; \|f\|_{M_m^p([\rho_1 \sigma, \; \rho_1 (\sigma+1)]\times [\rho_2 \sigma, \; \rho_2 (\sigma+1)])}\; \|1 \|_{M_{\frac{1}{m}}^q([\rho_1 \sigma, \; \rho_1 (\sigma+1)]\times [\rho_2 \sigma, \; \rho_2 (\sigma+1)])} \\
	& \leq  (\rho_1 \rho_2)^{\frac{2}{q}-1}  \|\Phi\|_{M_m^p([\rho_1 \sigma, \; \rho_1 (\sigma+1)]\times [\rho_2 \sigma, \; \rho_2 (\sigma+1)])}\\
	& \leq  A \; (\sigma+1)^{\frac{4}{p}-2}.
\end{align*}
Now the remaining part of the proof follows similarly as in \cite{cow83}.
\end{proof} 

\begin{theorem}\label{th2}
Let $g\in M_{\frac{1}{m}}^{1}(\R, A_{\alpha, \beta})$ be a non-zero window function  and  $1 \leq p, q \leq \infty$ with at least one of them finite.
Suppose that $f$ is a measurable function on $\R$ such that 
\begin{eqnarray}\label{eq5}
e^{a x^2} f \in M_m^p(\R, A_{\alpha, \beta}) \quad \text{and} \quad  e^{b (\lambda^2+\mu^2)} \W^{(\alpha, \beta)}_g(f) \in M_m^q(\R^2,  A_{\alpha, \beta} \otimes \sigma_{\alpha, \beta}),
\end{eqnarray}
for some constants $a, b > 0$. Then the following results hold:
\begin{enumerate}
\item[(i)] If $ab \geq \frac{1}{4}$, then $f = 0$ almost
everywhere.

\item[(ii)] If $ab < \frac{1}{4}$, then for all $t \in (b, \frac{1}{4a})$, the functions $f = E^{\alpha, \beta}_t$ satisfy the relations $(\ref{eq5})$.
\end{enumerate}
\end{theorem} 

\begin{proof}
We divide the proof into three steps.

Step 1: Assume that $ab > \frac{1}{4}$. The function
\[ \W^{(\alpha, \beta)}_g(f)(\lambda,\mu)=\int_\R f(s) \; \overline{g_{\lambda,\mu}^{(\alpha, \beta)}(-s)} \;  A_{\alpha, \beta}(s) \; ds, \quad  \quad \text{for any } \lambda,\mu \in \C, \]
is well defined, entire on $\C^2$, and satisfies the condition
\begin{align} \label{eq6}\nonumber
&|\W^{(\alpha, \beta)}_g(f)(\lambda,\mu)| \\\nonumber
& \leq   \int_{\R} |f(s)|\; |g_{\lambda,\mu}^{(\alpha, \beta)}(-s)|\; A_{\alpha, \beta} (s) ds  \\\nonumber
& =  \; e^{\frac{|\mathrm{Im}(\lambda)|^2+|\mathrm{Im}(\mu)|^2}{4a}}
\int_{\R}  e^{a s^2} |f(s)|\;  e^{- a \left(s^2 + \frac{|\mathrm{Im}(\lambda)|^2+|\mathrm{Im}(\mu)|^2}{4a^2}\right) }|g_{\lambda,\mu}^{(\alpha, \beta)}(-s)|\; A_{\alpha, \beta} (s) ds,\\ \nonumber& \quad  \text{so by H\"older's inequality}, 
\nonumber \\
& \leq   \; e^{\frac{|\mathrm{Im}(\lambda)|^2+|\mathrm{Im}(\mu)|^2}{4a}} \; \Big\Vert e^{a s^2} f \Big\Vert_{M_m^p(\R, A_{\alpha, \beta})} \; \Big\Vert  e^{- a \left(s^2 + \frac{|\mathrm{Im}(\lambda)|^2+|\mathrm{Im}(\mu)|^2}{4a^2}\right) }|g_{\lambda,\mu}^{(\alpha, \beta)}| \Big\Vert_{M_{\frac{1}{m}}^{p'}(\R, A_{\alpha, \beta})},
\end{align}
  where $p'$ is the conjugate exponent of $p$.  Since   $M_{\frac{1}{m}}^{p'}(\R, A_{\alpha, \beta})$ is invariant under translations and modulations, we get
\begin{align*}
\Big\Vert  e^{- a \left(s^2 + \frac{|\mathrm{Im}(\lambda)|^2+|\mathrm{Im}(\mu)|^2}{4a^2}\right) }|g_{\lambda,\mu}^{(\alpha, \beta)}| \Big\Vert_{M_{\frac{1}{m}}^{p'}(\R, A_{\alpha, \beta})}&\leq \Big\Vert   \tau_\lambda^{(\alpha, \beta)} \M^{(\alpha, \beta)}_\mu g\Big\Vert_{M_{\frac{1}{m}}^{p'}(\R, A_{\alpha, \beta})}\\
&\leq  \Vert     g \Vert_{M_{\frac{1}{m}}^{p'}(\R, A_{\alpha, \beta})}\leq  \Vert   g   \Vert_{M_{\frac{1}{m}}^{1}(\R, A_{\alpha, \beta})}<\infty.
\end{align*}
We consider the function $\Phi$ defined on $\C^2$ by
\begin{equation}\label{eq7}
\Phi(\lambda, \mu)=e^{\frac{\lambda^2+\mu^2}{4a}} \; \W^{(\alpha, \beta)}_g(f)(\lambda,\mu) .
\end{equation}
Then $\Phi$ is an entire function on $\C^2$ and using the relation (\ref{eq6}), we find  that there exists a constant $A$ for which 
\begin{equation}\label{eq8}
|\Phi(\lambda, \mu)| \leq A \;  e^{\frac{(\mathrm{Re}(\lambda))^2+(\mathrm{Re}(\mu))^2}{4a}}, \quad \text{for all } \lambda, \mu \in \C.
\end{equation}
In the following, we consider two cases.

(i) Let $q<\infty$. Using   $ab > \frac{1}{4}$ and the hypothesis (\ref{eq5}), we get 
\begin{align}\label{eq9}\nonumber
\| {\Phi}_{|\R^2} \|_{M_m^q(\R^2, A_{\alpha, \beta} \otimes \sigma_{\alpha, \beta})}&=\left\| e^{\frac{\lambda^2+\mu^2}{4a}} \; \W^{(\alpha, \beta)}_g(f)\;    \right\|_{M_m^q(\R^2, A_{\alpha, \beta} \otimes \sigma_{\alpha, \beta})} \\\nonumber
&=\left\| e^{b(\lambda^2+\mu^2)} \; \W^{(\alpha, \beta)}_g(f)\;  e^{\left(\frac{1}{4a} - b\right)(\lambda^2+\mu^2)} \right\|_{M_m^q(\R^2, A_{\alpha, \beta} \otimes \sigma_{\alpha, \beta})} \\&\leq \left\|  e^{b(\lambda^2+\mu^2)} \; \W^{(\alpha, \beta)}_g(f) \right\|_{M_m^q(\R^2, A_{\alpha, \beta} \otimes \sigma_{\alpha, \beta})} \leq A.
\end{align}
By applying   Lemma \ref{lem4} to the functions $\Phi(\lambda, \mu ), \; \Phi(-\lambda, -\mu), \; \overline{\Phi(\overline{\lambda}, \overline{\mu})}$  and  $\overline{\Phi(-\overline{\lambda}, -\overline{\mu})}$, we obtain that for all $\psi_1, \psi_2\in [0, 2\pi]$ and large $\sigma$
\[ \int_\sigma^{\sigma+1} \int_\sigma^{\sigma+1} |\Phi(\rho_1 e^{i \psi_1} , \rho_2 e^{i \psi_2})| \; d\rho_1d\rho_ 2 \leq  B (\sigma + 1)^{\frac{4}{q}-2},  \]
for some constant $B$. Now, applying Cauchy's integral formula for several complex variables (see \cite{sch}, Theorem 1.3.3), we get 
\[ |D^{(n)}\Phi(0)| \leq n! (2 \pi)^{-2} \int_0^{2 \pi}  \int_0^{2 \pi} |\Phi(\rho_1 e^{i \psi_1} , \rho_2 e^{i \psi_2})| \; \rho_1^{-n}\rho_2^{-n} d\psi_1d\psi_2.  \]
Hence, for large $\sigma$,
\begin{align}\label{eq10}\nonumber
|D^{(n)}\Phi (0)| &   \leq n! (2 \pi)^{-2} \int_0^{2 \pi}  \int_0^{2 \pi} \left( \int_\sigma^{\sigma+1} \int_\sigma^{\sigma+1}|\Phi(\rho_1 e^{i \psi_1} , \rho_2 e^{i \psi_2})| \; \rho_1^{-n}\rho_2^{-n}\;d\rho_1d\rho_2\right)  d\psi_1d\psi_2\\
& \leq  B n! \; \sigma^{-2n}  (\sigma + 1)^{\frac{4}{q}-2}.  
\end{align}
Let $\sigma \to \infty$. If $q \geq 2$, then there exists  a constant $B_1$ such that  $(\sigma + 1)^{\frac{2}{q}-1} \leq B_1$, 
and consequently $D^{(n)}\Phi (0)=0$ for $n \geq 1$. Thus $\Phi(\lambda, \mu)= D$, for some constant $D$. From (\ref{eq9}), $\Phi(\lambda,\mu) = 0$ for all $\lambda, \mu  \in \C$. Further, if $q < 2$, then $D^{(n)}\Phi(0)=0$ for $n \geq 2$. Hence $\Phi(\lambda, \mu)= C_1 \lambda+C_2 \mu+ D$, for some constants $C_1, C_2,$ and $D$. From (\ref{eq8}) and (\ref{eq9}), $\Phi(\lambda,\mu)= 0$ for all $\lambda, \mu  \in \C$. Therefore $\W^{(\alpha, \beta)}_g(f)(\lambda,\mu)=0$ for all $\lambda, \mu \in \R$.  Thus   $f=0$ almost everywhere on $\R$    by (\ref{eq17}).
  
(ii) Let $q = \infty$. As $ab > \frac{1}{4}$, then from (\ref{eq5}),  we have
\begin{equation}\label{eq11}
	\| {\Phi}_{|\R^2} \|_{M_m^\infty (\R^2, A_{\alpha, \beta} \otimes \sigma_{\alpha, \beta})}\leq \left\| e^{b({\lambda^2+\mu^2})} \; \W^{(\alpha, \beta)}_g(f)\;    \right\|_{M_m^\infty (\R^2, A_{\alpha, \beta} \otimes \sigma_{\alpha, \beta})} < \infty. 
\end{equation}
If   $q=\infty$, then we can  refined the estimate obtained in Lemma \ref{lem4} such  that $\max \{e^{2a} , (\sigma + 1)^{\frac{4}{q}-2} \}$ is replaced by $1$.  From (\ref{eq10}), we get 
\[ |D^{(n)}\Phi (0)| \leq   A \; n! \; \sigma^{-2n}.  \]
Then $D^{(n)}\Phi (0)=0$ for $n \geq 1,$ and this implies that  $\Phi(\lambda, \mu ) = C$ for all $\lambda, \mu  \in \C$ and for some constant $C$. Therefore 
$$e^{b (\lambda^2+\mu^2)} \W^{(\alpha, \beta)}_g(f)(\lambda,\mu)=C \; e^{(b-\frac{1}{4a}) (\lambda^2+\mu^2)}$$ 
for all $\lambda, \mu  \in \R$. Since $ab > \frac{1}{4}$, this function satisfies the relation (\ref{eq11}) implies that $C=0$. Thus from   (\ref{eq17}), we get $f = 0$ almost everywhere on $\R$.

Step 2: Assume that $ab = \frac{1}{4}$.

(a) If $q < \infty$, by  the same proof as for the point (i) of the first step, we  get  $f = 0$ almost everywhere on $\R$.

(b) Let $q = \infty$ and $1 \leq p < \infty$. We have $	\| {\Phi}_{|\R^2} \|_{M_m^\infty (\R^2, A_{\alpha, \beta} \otimes \sigma_{\alpha, \beta})}< \infty$. Then by the point (ii) of the first step, the relation (\ref{eq7}), and the property (\ref{eq05}) of the Gaussian kernel $E^{\alpha, \beta}_{\frac{1}{4a}}$, we obtain
\begin{equation}\label{eq12}
\W^{(\alpha, \beta)}_g(f)(\lambda,\mu)=C \; e^{-\frac{(\lambda^2+\mu^2)}{4a}}=C \; \W^{(\alpha, \beta)}_g( E^{\alpha, \beta}_{\frac{1}{4a}} ) (\lambda, \mu ), \quad \text{for all } \lambda, \mu  \in \R, 
\end{equation}
for some constant $C$. Thus, using  the injectivity of  $\W^{(\alpha, \beta)}_g$, we get 
\begin{equation}\label{eq14}
f(x)= C \;  E^{\alpha, \beta}_{\frac{1}{4a}} (x), \quad \text{a.e. } x \in \R.
\end{equation} 
By using the relations (\ref{eq06}) and (\ref{eq14}), we obtain
\[ \frac{2 C e^{\frac{\mu_1}{4a}} a^{\alpha+1}}{\Gamma(\alpha+1) \sqrt{B_{\alpha, \beta}(x)}} \leq e^{a x^2} f(x), \quad \text{for all } x \in \R. \]
For finite $p,$  using  the properties of the functions $A_{\alpha, \beta}$ and $B_{\alpha, \beta}$, we get  
\[ \left\| \frac{1}{\sqrt{B_{\alpha, \beta} (x)}} \right\|_{M_m^p(\R, \;A_{\alpha, \beta})} =\infty. \]
Moreover, from (\ref{eq5}) we have $\| e^{a x^2} f \|_{M_m^p(\R, \;A_{\alpha, \beta})} < \infty$, this is impossible unless $C = 0$. Then we obtain from (\ref{eq14}) that $f = 0$ almost everywhere on $\R$.
    
Step 3: Assume that $ab < \frac{1}{4}$. Let $t \in (b, \frac{1}{4a})$ and $f = E^{\alpha, \beta}_t$.  From the relation (\ref{eq06}), we get 
\[ K_1 e^{-\left(\frac{1}{4t}-a\right)x^2} \leq e^{ax^2} f(x) \leq K_2 e^{-\left(\frac{1}{4t}-a\right)x^2}, \quad \text{for all }  x \in \R, \] 
for some constants $K_1, \;K_2 > 0$. As $t < \frac{1}{4a}$, we deduce that  $e^{ax^2} f \in  {M_m^p(\R, A_{\alpha, \beta})} $.  Using the relation (\ref{eq04}), we get
\[ e^{b (\lambda^2+\mu^2)} \W^{(\alpha, \beta)}_g(f)(\lambda,\mu)=e^{-(t-b)(\lambda^2+\mu^2)}, \quad \text{for all }  \lambda, \mu  \in \R. \]
The condition $t > b$ and the inequality $|C_{\alpha, \beta} (\lambda)|^{-2} \leq k_2|\lambda|^{2 \alpha +1}$ at infinity (see \cite{tri97}, page 157) imply that $e^{b (\lambda^2+\mu^2)} \W^{(\alpha, \beta)}_g(f)\in M_m^q(\R^2, A_{\alpha, \beta} \otimes \sigma_{\alpha, \beta})$. This completes the proof.
\end{proof}

\subsection{Hardy's theorem for the windowed Opdam--Cherednik transform}\label{sec5}

Here, we obtain an analogue of the classical Hardy's theorem for the windowed Opdam--Cherednik transform on weighted modulation spaces associated with this transform. In particular, we determine the functions $f$ satisfying the relations (\ref{eq5}) in the special case $p = q = \infty$.

\begin{theorem}
Let $g\in M_{\frac{1}{m}}^{1}(\R, A_{\alpha, \beta})$ be a non-zero window function and  $f$ be a measurable function on $\R$ such that
\begin{equation}\label{eq15}
e^{a x^2} f \in M_m^\infty (\R, A_{\alpha, \beta}) \quad \text{and} \quad  e^{b (\lambda^2+\mu^2)} \W^{(\alpha, \beta)}_g(f) \in M_m^\infty (\R^2,  A_{\alpha, \beta} \otimes \sigma_{\alpha, \beta}),
\end{equation}
for some constants $a, b > 0$. Then
\begin{enumerate}
\item[(i)] If $ab > \frac{1}{4}$, we have $f = 0$ almost
everywhere.

\item[(ii)] If $ab = \frac{1}{4}$, the function $f$ is of the form $f = C  E^{\alpha, \beta}_{\frac{1}{4a}} $, for some real constant $C$.

\item[(iii)] If $ab < \frac{1}{4}$, there are infinitely many nonzero functions $f$ satisfying the conditions $(\ref{eq15})$.
\end{enumerate}
\end{theorem}
\begin{proof}
(i) If $ab > \frac{1}{4}$,  the result follows from   point (ii) of the first step of the proof of Theorem \ref{th2}.

(ii) If $ab = \frac{1}{4}$ and $\| e^{b (\lambda^2+\mu^2)} \W^{(\alpha, \beta)}_g(f) \|_{M_m^\infty (\R^2,  A_{\alpha, \beta} \otimes \sigma_{\alpha, \beta})} < \infty$, then from    (\ref{eq14}) and Step 2(b) of the proof of Theorem \ref{th2}, we have $f = C  E^{\alpha, \beta}_{\frac{1}{4a}} $, for some real constant $C$. Using the property    $B_{\alpha, \beta} (x) \geq 1$, from relations (\ref{eq06}) and (\ref{eq14}), we get  
\[ e^{a x^2} f(x) \leq \frac{2 C e^{\frac{\mu_2}{4a}} a^{\alpha+1}}{\Gamma(\alpha+1) \sqrt{B_{\alpha, \beta}(x)}}, \quad \text{for all } x \in \R. \]
Moreover, from (\ref{eq15}) we have $\| e^{a x^2} f \|_{M_m^\infty(\R, \;A_{\alpha, \beta})} < \infty$, this is impossible unless $f = C  E^{\alpha, \beta}_{\frac{1}{4a}} $. This completes the  result of point (ii).

(iii) If $ab < \frac{1}{4}$, the functions $f = E^{\alpha, \beta}_t$, $t \in (b, \frac{1}{4a})$, satisfy the conditions $(\ref{eq15})$. This completes the proof of the theorem.
\end{proof}  

\subsection{Morgan's theorem for the windowed Opdam--Cherednik transform}\label{sec6}

The aim of this subsection is to prove an $M_m^p$ -- $M_m^q$-version of Morgan's theorem for the windowed Opdam--Cherednik transform. Before we prove the main result of this subsection, we first need the following lemma. 

\begin{lemma}[\cite{far03}, Lemma 2.3] \label{lem1}
	Suppose that $\rho \in (1, 2)$, $q \in [1, \infty]$, $\eta > 0$, $M > 0$ and $ B > \eta \sin \frac{\pi}{2} (\rho-1)$. If $\Phi$ is an entire function on $\C^2$ satisfying the conditions
	\begin{enumerate}
		\item[(i)] $|\Phi(x + iy, u + iv)| \leq M e^{\eta (|y|^\rho+|v|^\rho)},\;$ for any $x,y, u, v \in \R$, 
		\item[(ii)] $ e^{B (|x|^\rho+|u|^\rho)} {\Phi}_{|\R^2}  \in L^q(\R^2) $, 
	\end{enumerate}
	then $\Phi = 0$.
\end{lemma}

As an application of  the above lemma,  in the following,  we obtain a version of the Phragm{\'e}n--Lindl{\"o}f type result for the  weighted modulation spaces on $\R^2$.
\begin{lemma}\label{lem2}
	Suppose that $\rho \in (1, 2)$, $q \in [1, \infty)$, $\eta > 0$, $M > 0$ and $B > \eta \sin \frac{\pi}{2} (\rho-1)$. If $\Phi$ is an entire function on $\C^2$ satisfying the conditions
	\begin{enumerate}
		\item[(i)] $|\Phi(x + iy, u + iv)| \leq M e^{\eta (|y|^\rho+|v|^\rho)}$, for any
		$x,y, u, v \in \R$,
		\item[(ii)] $ e^{B (|x|^\rho+|u|^\rho)} {\Phi}_{|\R^2}  \in M_m^q(\R^2, A_{\alpha, \beta}\otimes \sigma_{\alpha, \beta}
		) $,
	\end{enumerate}
	then $\Phi = 0$.
\end{lemma}
\begin{proof}
	Let $R > 0$ be such that 
	\[B > \eta ((R+1)/R)^\rho \sin \frac{\pi}{2} (\rho-1).\] 
	Let us  consider the entire function $F$ on $\C^2$   by 
	\[ F(z_1, z_2)=\int_R^{R+1} \int_R^{R+1} \Phi(t_1z_1, t_2z_2)\; dt_1 dt_2.\]
	Then  using  Cauchy's integral formula for several complex variables (see \cite{sch}, Theorem 1.3.3), we obtain that the derivatives of $F$ satisfy the condition 
	\begin{align*}
		D^{(n)}F(0) &= \left[\left( (R + 1)^{n+1} - R^{n+1}\right) /(n+1)\right]^2 \;D^{(n)}\Phi(0), \quad  \text{for any $n\in \mathbb{N}$}.  
	\end{align*} 
	Therefore, $\Phi =0$ if and only if $F =0$. By assumption (i), we get 
	\begin{equation}\label{eq01}
		|F(x+iy, u+iv)| \leq  M \;e^{{(R+1)}^{\rho} \eta(|y|^\rho+|v|^\rho)},\mathrm{\;for \;any\;} x,y, u, v \in \R.
	\end{equation}
	Let $x, u \in \R \setminus \{0\}$. Then using the change of variables $x_1 = xt_1$  and $u_1 = ut_2$, we obtain
	\[ F(x, u)= \frac{1}{x u} \int_{Rx}^{(R+1)x} \int_{Ru}^{(R+1)u} \Phi(x_1, u_1) \;dx_1 du_1.\]
	Hence, 
	\begin{align*}
		| F(x, u)| &\leq \frac{1}{|x||u|} \int_{Rx}^{(R+1)x} \int_{Ru}^{(R+1)u} |\Phi(x_1, u_1)| \;dx_1 du_1\\&
		= \frac{1}{|x||u|} \int_{Rx}^{(R+1)x} \int_{Ru}^{(R+1)u} \;e^{B (|x_1|^\rho+|u_1|^\rho)}\;e^{-B (|x_1|^\rho+|u_1|^\rho)} |\Phi(x_1, u_1)|  \;dx_1 du_1\\
		&\leq  \frac{1}{|x||u|}\;e^{-B R^\rho(|x|^\rho+|u|^\rho)}   \int_{Rx}^{(R+1)x} \int_{Ru}^{(R+1)u} \;e^{B (|x_1|^\rho+|u_1|^\rho)}|\Phi(x_1, u_1)|  \;dx_1 du_1.
	\end{align*}
	Using  H\"older's inequality and the relation (\ref{eq31}), we get
	\begin{align*}
		| F(x, u)| &\leq \frac{1}{|x||u|}\;e^{-B R^\rho(|x|^\rho+|u|^\rho)}    \;\Vert e_B \Phi \Vert_{M_m^q(\R^2)}
		\; \Vert 1 \Vert_{M_{\frac{1}{m}}^{q'}([Rx, (R+1)x]\times [Ru, (R+1)u])}\\
		&\leq \frac{1}{|x||u|}\;e^{-B R^\rho(|x|^\rho+|u|^\rho)}    \;\Vert e_B \Phi \Vert_{M_m^q(\R^2, A_{\alpha, \beta}\otimes \sigma_{\alpha, \beta})}
		\; \Vert 1 \Vert_{M_{\frac{1}{m}}^{q'}([Rx, (R+1)x]\times [Ru, (R+1)u])}, 
	\end{align*} 
	where $e_B(x_1, u_1)=\;e^{B (|x_1|^\rho+|u_1|^\rho)} $ and $q'$ is the conjugate exponent of $q$. Since 
	$$\Vert 1 \Vert_{M_{\frac{1}{m}}^{q'}([Rx, (R+1)x]\times [Ru, (R+1)u])}\leq C |x|^{\frac{2}{q'}}|u|^{\frac{2}{q'}}$$ 
	for some constant $C>0$, we have

	\[| F(x, u)|  \leq \frac{C}{|x|^{1-2/q'}|u|^{1-2/q'}} \; e^{-B R^\rho(|x|^\rho+|u|^\rho)}    \;\Vert e_B \Phi \Vert_{M_m^q(\R^2, A_{\alpha, \beta}\otimes \sigma_{\alpha, \beta})}.\]
	Since $F$ is continuous on $\R^2$, using  assumption (ii), we obtain 
	\begin{equation}\label{eq02}
		e^{B R^\rho(|x|^\rho+|u|^\rho)}    F_{|\R^2} \in L^\infty(\R^2).
	\end{equation}
	Using the inequalities (\ref{eq01}) and (\ref{eq02}), and applying Lemma \ref{lem1} for $q=\infty$ to $F$, we get  $F=0$, thus $\Phi=0$. This completes the proof of the lemma.
\end{proof}

\begin{theorem}
	Let $g\in M_{\frac{1}{m}}^{1}(\R, A_{\alpha, \beta})$ be a non-zero window function,  $ p \in [1, \infty]$, $ q \in [1, \infty)$, $a>0$, $b>0$, and let $\alpha, \beta$ be positive real numbers satisfying $\alpha >2$ and $1/\alpha+1/\beta=1$. Suppose that $f$ is a measurable function on $\R$ such that  
	\[ e^{a|x|^\alpha} f \in M_m^p(\R, A_{\alpha, \beta}) \quad \text{and} \quad  e^{b (|\lambda|^\beta+|\mu|^\beta)} \W^{(\alpha, \beta)}_g(f) \in M_m^q (\R^2,  A_{\alpha, \beta} \otimes \sigma_{\alpha, \beta}). \]
	If 
	\[ (a \alpha)^{1/\alpha} (b \beta)^{1/\beta} > \left( \sin \left( \frac{\pi}{2}(\beta -1 ) \right) \right)^{1/\beta},  \]
	then $f = 0$. 
\end{theorem}

\begin{proof}
	Assume that $f$ is a measurable function on $\R$ such that
	\begin{equation}\label{eq2}
		e^{a|x|^\alpha} f \in M_m^p(\R, A_{\alpha, \beta})
	\end{equation}
	and
	\begin{equation}\label{eq3}
		e^{b (|\lambda|^\beta+|\mu|^\beta)} \W^{(\alpha, \beta)}_g(f) \in M_m^q (\R^2,  A_{\alpha, \beta} \otimes \sigma_{\alpha, \beta}).
	\end{equation}
	To prove that the windowed Opdam--Cherednik transform of $f$ satisfies the conditions (i) and (ii) of Lemma \ref{lem2}, 	we use conditions (\ref{eq2}) and (\ref{eq3}), and   we deduce that $f=0$ almost everywhere.
	
	The function
	\[ \W^{(\alpha, \beta)}_g(f)(\lambda,\mu)=\int_\R f(s) \; \overline{g_{\lambda,\mu}^{(\alpha, \beta)}(-s)} \;  A_{\alpha, \beta}(s) \; ds\] 
	is well defined, entire on $\C^2$, and satisfies the condition
	\begin{align}\label{eq33}\nonumber
		|\W^{(\alpha, \beta)}_g(f)(\lambda,\mu)| & \leq   \int_{\R} |f(s)|\; |g_{\lambda,\mu}^{(\alpha, \beta)}(-s)|\; A_{\alpha, \beta} (s) ds \\\nonumber
		& \leq \; \Big\Vert e^{a|s|^\alpha} f \Big\Vert_{M_m^p(\R, A_{\alpha, \beta})} \; \Big\Vert e^{-a|s|^\alpha} |g_{\lambda,\mu}^{(\alpha, \beta)} | \Big\Vert_{M_{\frac{1}{m}}^{p'}(\R, A_{\alpha, \beta})},   \text{by H\"older's inequality}, \\
		& \leq   C\; \Big\Vert e^{-a|s|^\alpha} |g_{\lambda,\mu}^{(\alpha, \beta)} | \Big\Vert_{M_\frac{1}{m}^{p'}(\R, A_{\alpha, \beta})}, \quad \text{by } (\ref{eq2}),
\end{align}
where $C$ is a constant and $p'$ is the conjugate exponent of $p$. 
	
Let
\[C \in I= \left( (b \beta)^{-1/\beta} \left( \sin \left( \frac{\pi}{2}(\beta -1 ) \right) \right)^{1/\beta}, \; (a \alpha)^{1/\alpha}  \right) . \]
Now 
\begin{align}\label{eq32}\nonumber
&	\Big\Vert e^{-a|s|^\alpha} |g_{\lambda,\mu}^{(\alpha, \beta)} | \Big\Vert_{M_\frac{1}{m}^{p'}(\R, A_{\alpha, \beta})}\\\nonumber&=	\Big\Vert e^{-a|s|^\alpha} e^{-|s| (|\operatorname{Im}(\lambda)|^\beta+|\operatorname{Im}(\mu)|^\beta)^{\frac{1}{\beta} }}  e^{|s| (|\operatorname{Im}(\lambda)|^\beta+|\operatorname{Im}(\mu)|^\beta)^{\frac{1}{\beta} }} |g_{\lambda,\mu}^{(\alpha, \beta)} | \Big\Vert_{M_\frac{1}{m}^{p'}(\R, A_{\alpha, \beta})}\\
		&\leq \Big\Vert e^{-a|s|^\alpha} e^{|s| (|\operatorname{Im}(\lambda)|^\beta+|\operatorname{Im}(\mu)|^\beta)^{\frac{1}{\beta} }} |g_{\lambda,\mu}^{(\alpha, \beta)} | \Big\Vert_{M_\frac{1}{m}^{p'}(\R, A_{\alpha, \beta})}.
	\end{align}
	Applying the convex inequality
	\[ |ty| \leq \left(\frac{1}{\alpha}\right)|t|^\alpha + \left(\frac{1}{\beta}\right)|y|^\beta \]
	to the positive numbers $C|t|$ and $|y|/C$, we get 
	\[ |ty| \leq \frac{C^\alpha}{\alpha} \;|t|^\alpha + \frac{1}{\beta C^\beta} \;|y|^\beta, \]
	and thus

	\begin{align}\label{eq34}\nonumber
		&	\Big\Vert e^{-a|s|^\alpha} e^{|s| (|\operatorname{Im}(\lambda)|^\beta+|\operatorname{Im}(\mu)|^\beta)^{\frac{1}{\beta} }} |g_{\lambda,\mu}^{(\alpha, \beta)} | \Big\Vert_{M_\frac{1}{m}^{p'}(\R, A_{\alpha, \beta})}\\
		&\leq e^{\frac{ |\operatorname{Im}(\lambda)|^\beta+|\operatorname{Im}(\mu)|^\beta }{\beta C^\beta }}
		\Big\Vert e^{-(a-C^\alpha/\alpha )|s|^\alpha}    |g_{\lambda,\mu}^{(\alpha, \beta)}| \Big\Vert_{M_\frac{1}{m}^{p'}(\R, A_{\alpha, \beta})}.
	\end{align}
	Since $C \in I$ , it follows that $a > C^\alpha / \alpha$, and thus   
	\begin{align}\label{eq35}\nonumber
		\Big\Vert e^{-(a-C^\alpha/\alpha )|s|^\alpha}    |g_{\lambda,\mu}^{(\alpha, \beta)} | \Big\Vert_{M_\frac{1}{m}^{p'}(\R, A_{\alpha, \beta})}  &\leq \Big\Vert   \tau_\lambda^{(\alpha, \beta)} \M^{(\alpha, \beta)}_\mu g\Big\Vert_{M_{\frac{1}{m}}^{p'}(\R, A_{\alpha, \beta})}\\
		&\leq  \Vert     g \Vert_{M_{\frac{1}{m}}^{p'}(\R, A_{\alpha, \beta})}\leq  \Vert   g   \Vert_{M_{\frac{1}{m}}^{1}(\R, A_{\alpha, \beta})}<\infty.
	\end{align}
	From  (\ref{eq32}), (\ref{eq34}) and (\ref{eq35}), we get 
	\[\Big\Vert e^{-a|s|^\alpha} |g_{\lambda,\mu}^{(\alpha, \beta)} | \Big\Vert_{M_\frac{1}{m}^{p'}(\R, A_{\alpha, \beta})} < \infty. \]
	Moreover, 
	\begin{equation}\label{eq4}
		|\W^{(\alpha, \beta)}_g(f)(\lambda,\mu)|  \leq \text{Const.}\; e^{\frac{ |\operatorname{Im}(\lambda)|^\beta+|\operatorname{Im}(\mu)|^\beta }{\beta C^\beta }} \quad \text{for any } \lambda, \mu  \in \C.
	\end{equation}
	Using the  condition (\ref{eq3}) and inequality (\ref{eq4}), we obtain that
	the function $\Phi(\lambda, \mu)=\W^{(\alpha, \beta)}_g(f)(\lambda,\mu)$ satisfies the assumptions (i) and (ii) of Lemma \ref{lem2} with $\rho = \beta$, $\eta=1/(\beta C^\beta)$, and $B=b$. The condition $C \in I$ implies the inequality 
	\[ b > \frac{1}{\beta C^\beta}  \sin \left( \frac{\pi}{2}(\beta -1 ) \right),  \]
	which gives $\W^{(\alpha, \beta)}_g(f)=0$ by Lemma \ref{lem2}, then $f=0$ by (\ref{eq17}).  This completes the proof of the theorem.
\end{proof}

\section*{Acknowledgments}
The first author gratefully acknowledges the support provided by IIT Guwahati, Government of India. 
The second author is deeply indebted to Prof. Nir Lev for several fruitful discussions and generous comments. The authors wish to thank the anonymous referees for their helpful comments and suggestions that helped to improve the quality of the paper.

\section*{Declarations} 
\textbf{Conflict of interest} No potential conflict of interest was reported by the authors.

\end{document}